\documentclass[11pt,reqno]{amsart}
\usepackage{a4,amssymb,amsthm,amscd,amsmath,verbatim,url,enumerate,mathdots, hyperref,
cleveref,mathtools}
\usepackage[pdftex,dvipsnames]{xcolor}
\usepackage{fancyhdr}
\usepackage{amsmath}

\title{On the Isotropy Groups of Non-Invertible Simple Derivations}

\author{Sumit Chandra Mishra}
\address{Indian Institute of Technology Indore, Simrol, Khandwa Road, Indore 453 552 India}
\email{sumitcmishra@gmail.com, sumitcmishra@iiti.ac.in}

\author{Dibyendu Mondal}
\address{Indian Institute of Technology Indore, Simrol, Khandwa Road, Indore 453 552 India}
\email{mdibyendu07@gmail.com, mdibyendu@iiti.ac.in}

\author{Pankaj Shukla}
\address{Indian Institute of Technology Indore, Simrol, Khandwa Road, Indore 453 552 India}
\email{Pankajshuklashvm23@gmail.com, phd2201141002@iiti.ac.in}

\date{\today}

\renewcommand{\deg}{\mathsf{deg}}

\renewcommand{\setminus}{\smallsetminus}

\renewcommand{\leq}{\leqslant}
\renewcommand{\geq}{\geqslant}

\swapnumbers

\newtheorem*{thm*}{Theorem}

\newtheorem{thm}{Theorem}
\numberwithin{thm}{section}

\newtheorem{cor}[thm]{Corollary}
\newtheorem{conj}[thm]{Conjecture}

\newtheorem{lem}[thm]{Lemma}

\theoremstyle{definition}

\newtheorem{remark}[thm]{Remark}
\numberwithin{equation}{thm}
\renewenvironment{proof}{\par\noindent {\em Proof:}}{\hfill$\Box$\medskip}
\theoremstyle{plain}

\begin{document}
\begin{abstract}

Let $k$ be a field of characteristic zero, and let $i$ and $n$ be positive integers with $i\geq 2$ and $n>i$. Consider a non-invertible $k$-derivation $d_i$ of the polynomial ring $k[x_1,\ldots,x_i]$. Let $d_n$ be an extension of $d_i$ to a derivation of $k[x_1,\ldots, x_n]$ such that $d_n(x_j)\in k[x_{j-1}]\setminus k$ for each $j$ with $i+1 \leq j\leq n$. In this article, we undertake a systematic study of the isotropy groups associated with such non-invertible derivations. We establish sufficient conditions on $d_i$ under which the isotropy group of the non-invertible simple derivation $d_n$ is conjugate to a subgroup of translations.
\end{abstract}

\maketitle

\section{Introduction}
    Let $k$ be a field of characteristic zero and let $n\in \mathbb{N}$. We denote by $R_n$ the polynomial ring in $n$ variables over the field $k$, that is, $R_n=k[x_1,\ldots,x_n]$ where $x_1,\ldots,x_n$ are indeterminates over $k$. A \emph{$k$-derivation} on $R_n$ is a $k$-linear map $d : R_n \rightarrow R_n$ such that $d(fg) = g\,d(f) + f\, d(g)$ for all $f, g \in R_n$. 
    By Aut($R_n$) we denote the $k$-automorphism group of $R_n$. 

    Let $d$ be a $k$-derivation of $R_n$. We denote by Aut$(R_n)_{d}$ the set of $k$ -automorphisms of $R_n$ which commute with $d$. Then Aut$(R_n)_{d}$ is a subgroup of Aut$(R_n)$, and is called the \emph{isotropy group} of $d$. The motivation to study these groups mainly comes from the problem of classifying and understanding affine $k$-algebras. In recent years, many authors have studied the isotropy groups of specific types of derivations, particularly for locally nilpotent derivations and simple derivations; see \cite{Bal21,BN,MP17,MPR24,Y22,YD24} etc.

    A $k$-derivation $d$ of $R_n$ is called \emph{simple} if $R_n$ does not have any proper non-zero ideal $I$ such that $d(I)\subseteq I$. Simple derivations have applications in Ore extensions \cite{g89}, and in the construction of non-commutative simple rings and non-holonomic irreducible modules over Weyl algebras (\cite{g89} and \cite{c07}). In this article, we mainly focus on non-invertible simple derivations of $R_n$, that is, simple derivations which do not contain any units in their images. One motivation for studying such derivations is to explore the skew polynomial rings $R[x,d]$, a non-commutative structure built from a ring $R$ equipped with a simple derivation $d$, where $d(R)$ contains no units of $R$. In general, very few examples of simple derivations are known. For details on examples of classes of simple derivations; see \cite{ g09,k14,n08,ap23} etc. For $n\geq 3$, most of the known examples of simple derivations come from a result of Shamsuddin. He provided a criterion to extend a simple derivation of a ring $A$ to a simple derivation of the polynomial ring $A[t]$, where $t$ is an indeterminate. In this article, we study the isotropy groups of such extended derivations. 

    In \cite{MP17}, Mendes and Pan proved that Aut($R_2$)$_{d}$ is the trivial group for any simple derivation $d$ on $R_2=k[x_1,x_2]$. Shamsuddin type derivations are a well-studied class of derivations. Indeed, multiple criteria are known for such derivations to be simple; see \cite{l08,n94}. One such criterion is in terms of isotropy groups. For $n\geq 2$, it is known that a derivation of Shamsuddin type is simple if and only if its isotropy group is trivial; see \cite{YD24}. In general, D. Yan in \cite{Y22} proposed the following conjecture.

    \begin{conj}\label{conj}
    If $d$ is a simple derivation of $R_n$ then $\text{Aut}(R_n)_{d}$ is conjugate to a subgroup of translations.
    \end{conj}

    Some examples of simple derivations that yield positive answers to \Cref{conj} can be found in \cite{HY23}. Recently, in \cite{MPR24}, the authors proved that the connected component of the isotropy group of a simple derivation of $R_n$ is an algebraic group.

    Note that the Shamsuddin type derivations contain units in their images. In fact, most of the known general results regarding isotropy groups of simple derivations $d$ are such that the image of $d$ contains units. In this article, we present a systematic study of the isotropy groups associated with certain classes of non-invertible derivations of the polynomial ring  $R_n$ for $n\geq 3$. 
    
    Let $i\geq 2$ and $n>i$ be positive integers. Suppose $d_i$ is a derivation of $R_i=k[x_1,\ldots,x_i]$ such that the image of $d_i$ does not contain any non-zero elements of $k[x_i]$ of degree $\leq l$, where $l\in \mathbb{N} \cup \{\infty\}$. Let $d_n:=d_i+g_i(x_i)\partial_{x_{i+1}}+\dots+g_{n-1}(x_{n-1})\partial_{x_n}$  be a $k$-derivation of $R_n$, where $g_j(x_j)\in k[x_j]\setminus k$ for all $i\leq j\leq n-1$ and $\deg_{x_{i}} g_i \leq l$. If the isotropy group of $d_i$ is trivial, then we prove that, under some suitable conditions on $d_i$, the isotropy group of $d_n$ is isomorphic to a subgroup of translations; see \Cref{thm-isotropy}. By a result of Shamsuddin (\cite{sam}), the extended derivation $d_n$ is simple if the derivation 
    $d_i$ is simple. Thus, as a consequence of \Cref{thm-isotropy}, we obtain a sufficient criterion for determining when such extended derivations have isotropy group conjugate to a subgroup of translations, thereby identifying a broad class of simple derivations for which \Cref{conj} holds. 
    As a corollary, we obtain Theorem 3.4 of \cite{Y24} and Theorem 4.3 of \cite{MMS25}.

    In the last section, we provide a class of non- invertible simple derivations $d_2$ of $R_2=k[x_1,x_2]$. 
    In fact, we show that the image of $d_2$ does not contain non-zero elements of $k[x_2]$ (see \Cref{lemma1}), thereby generalizing a result of \cite{k12}. Thus we obtain a class of derivations $d_2$ which satisfy the hypothesis of \Cref{thm-isotropy}.

\section{Isotropy groups}\label{iso-section}
    Let $k$ be a field of characteristic zero. Let $i\geq 2$ be a positive integer. Let $d_i$ be a $k$-derivation of $R_i=k[x_1,\ldots,x_i]$. The derivation $d_i$ is uniquely determined by the values $d_i(x_j)$ for all $1\leq j\leq i$. Let $n\in \mathbb{Z}$ such that $n>i$, and let $R_n:=R_i[x_{i+1},\ldots, x_n]$ be the polynomial ring in $n$ variables over $k$. Note that $R_i\subset R_n$. 

    Next, we define a $k$-derivation of $R_n$ which is induced from the $k$-derivation $d_i$. Let 
        \begin{align}\label{eqn i}
		d_n:=d_i+g_i(x_i)\partial_{x_{i+1}}+\dots+g_{n-1}(x_{n-1})\partial_{x_n},
        \end{align} 
    where $g_j(x_j)\in k[x_j]\setminus k$ for all $i\leq j\leq n-1$. Then $d_n$ is a $k$-derivation of $R_n$. Note that $d_n\vert_{R_i}=d_i$ and $d_i(x_j)=0$ for $i+1\leq j\leq n$.
    
    Let $n\geq j>i$, we denote by $d_j$ the $k$-derivation $d_j:=d_i+g_i(x_i)\partial_{x_{i+1}}+\dots+g_{j-1}(x_{j-1})\partial_{x_j}$ on $R_j:=R_i[x_{i+1},\ldots x_j]$. Then $d_j$ is a $k$-derivation of $R_j$.

    \begin{lem}\label{lemma no units}
        Let $i,n$, $R_i$, $R_n$, $d_i$ and $d_n$ be as above. Let $\deg_{x_i}g_i(x_i)\leq l$ for some $l\in \mathbb{N} \cup \{\infty\}$. Furthermore, assume that $d_i$ satisfies the following property: If $d_i(r)=g(x_i)$ for some $r\in R_i$ and $g(x_i)\in k[x_i]$ with $\deg_{x_i}g\leq l$, then $r\in k$ and $g=0$. 
        
        Let $j\in \{i+1,\ldots, n\}$. Now, if $d_n(f)=h(x_j)\in k[x_j]$ for some $f\in R_j$, then $f\in k$ and $h=0$.
        
    \end{lem}

    \begin{proof} Note that for any $f\in R_j$, for $i+1\leq j\leq n$, we have $d_n(f)=d_j(f)$.
    
        (i) We first prove the result for $j=i+1$. Recall that $d_{i+1}=d_i+g_i(x_i)\partial_{x_{i+1}}$ where $g_i(x_i)\in k[x_i]\setminus k$ with $\deg_{x_i}g_i(x_i)\leq l$. Let $f\in R_{i+1}$ such that $d_{i+1}(f)=h(x_{i+1})$ for some $h(x_{i+1})\in k[x_{i+1}]$.  If $f=0$, we are done. Otherwise $f\in R_{i+1}\backslash\{0\}$, and we can write $f=\sum_{j=0}^{t}f_jx_{i+1}^j$ where $t\geq 0$ and $f_j\in R_i$ for all $0\leq j\leq t$ and $f_t\neq 0$. 
        
        If $t=0$, that is $f\in R_{i}$, then $d_{i+1}(f)=d_{i}(f)$. Now, $d_{i+1}(f)=d_{i}(f)=h(x_{i+1})$ and $d_i(f)\in R_i$ implies that $h(x_{i+1})\in k$. Since $d_{i}(f)\in k$, using the property of $d_i$, we get that $f\in k$ and $h=0$.
        
		Now, we suppose $t\geq 1$, Then
		\begin{align}\label{onlyeqn}
			d_{i+1}(f)&=d_{i+1}(\sum_{j=0}^{t}f_jx_{i+1}^j)\\ \nonumber
			&=x_{i+1}^t	d_{i+1}(f_t)+x_{i+1}^{t-1}(tg_{i}(x_{i})f_t+d_{i+1}(f_{t-1}))+ T,
		\end{align} 
        where $f_t\neq 0$ and $T \in R_{i}[x_{i+1}]$ with $\deg_{x_{i+1}}(T) <t-1$.
        
		Now, $d_{i+1}(f)=h(x_{i+1})$ in \Cref{onlyeqn} implies that $d_{i+1}(f_{t})\in k$ and $tg_{i}(x_{i})f_t+d_{i+1}(f_{t-1})\in k$. As for all $j$, $ 0\leq j\leq t$, we have $d_{i+1}(f_{j})=d_{i}(f_{j})$ since $f_j\in R_{i}$. In particular, $d_{i}(f_{t})\in k$. By the assumption on $d_i$, we get that $f_t\in k^{\ast}$. Also, since $tf_tg_{i}(x_{i})+d_{i}(f_{t-1})\in k$ and $f_t\in k^{\ast}$, it follows that $d_i(f_{t-1})\in k[x_{i}]$ with $\deg_{x_i}(d_i(f_{t-1}))\leq l$. Again from the assumption on $d_i$, we have $f_{t-1}\in k$ and thus $d_{i+1}(f_{t-1})=0$. Then $tf_tg_{i}(x_{i})\in k$, which is a contradiction as $g_i(x_i)\in k[x_i]\setminus k$. 

        (ii) Note that it is enough to prove the result for $j=i+2$. Replacing $i$ by $i+1$ in (i), the proof proceeds similarly. 
    \end{proof}  

    Next, suppose $d_i(x_{j})=h_j(x_1,\ldots ,x_i)$ where $h_j\in R_i\backslash\{0\}$ for all $j=1,\ldots ,i$. Then 
    \begin{align*}
        d_i=h_1(x_1,x_2,\dots,x_i)\partial_{x_1}+h_2(x_1,x_2,\dots,x_i)\partial_{x_2}+\dots+h_i(x_1,x_2,\dots,x_i)\partial_{x_i}.
    \end{align*}    

    We can write
    \begin{align}\label{matrix monomial}
        h_j(x_1,x_2,\dots,x_i)=\alpha_jx_1^{u_{j1}}x_2^{u_{j2}}\dotsm x_i^{u_{ji}} + O_j(x_1,\ldots ,x_i)
    \end{align}
    such that $\alpha_j\in k^{\ast}$, $\sum_{l=1}^{i}u_{jl}$ is the total degree of $h_j$ and $O_j\in R_i$ (The motivation to write such an expression of $h_j$ is to choose one of the monomial from the total degree homogeneous part of $h_j$). Next, we construct a $i\times i$ matrix $A_{d_i}$ over $\mathbb{Z}$ corresponding to $d_i$. We construct the matrix $A_{d_i}$ from the power of the chosen monomials of $h_j$. Let 
    \begin{align}\label{matrix def}
        A_{d_i}:=[u_{jq}]_{1\leq j,q\leq i}
    \end{align}
    where $u_{jq}$ are from \Cref{matrix monomial}. Note that the matrix $A_{d_i}$ does not correspond uniquely to the derivation $d_i$, as the choice of the monomial from $h_j$ is not unique. 

    Let $M$ be a $i\times i$ matrix over $\mathbb{Z}$, $I$ denote the $i\times i$ identity matrix over $\mathbb{Z}$ and $Y$ be the $i\times 1$ vector of indeterminates over $\mathbb{Z}$. We say that $M$ satisfies the condition ($\ast$) if
    \begin{align}\label{$*$}
	\text{the system of inequalities}\,\, (M-I)Y\leq 0  \,\,\, \text{has no non-trivial solution in $\mathbb{Z}_{\geq 0}^i$}. 
    \end{align}

    Next, we show that under suitable conditions, restrictions of elements of Aut$_{d_n}(R_n)$ to $R_i$ are in Aut$_{d_i}(R_i)$.

    \begin{lem}\label{restriction_to_R_i}
	Let $i,n$, $R_i$, $R_n$, $d_i$ and $d_n$ be as above. Let $\rho\in Aut(R_n)_{d_n}$. Suppose that there exists a matrix $A_{d_i}$ corresponding to $d_i$, as in \Cref{matrix def}, satisfying the condition $(\ast)$. Then $\rho|_{R_i}\in Aut(R_i)_{d_i}$.
    \end{lem}

    \begin{proof}
    Let $A_{d_i}$ be the matrix corresponding to $d_i$ satisfying the condition ($\ast$). Let ${[u_{jq}]}_{1\leq j,q\leq i}$ denote the entries of $A_{d_i}$. Then, for all $1\leq j\leq i$, we can write 
    \begin{align}\label{eqn2.3}
        h_j(x_1,x_2,\dots,x_i):=d(x_{j})=\alpha_jx_1^{u_{j1}}x_2^{u_{j2}}\dotsm x_i^{u_{ji}} + O_j(x_1,\ldots ,x_i)
    \end{align}
    where $\alpha_j\in k^{\ast}$, $\sum_{l=1}^{i}u_{jl}$ is the total degree of $h_j$ and $O_j\in R_i$.
    
        Let $\rho\in Aut(R_n)_{d_n}$. For all $1\leq j\leq i$, let us write 
        \begin{align}\label{eqn2.1}
            \rho(x_j)=\sum_{m=0}^{t_j}f_{jm}x_n^m,
        \end{align}
         where $f_{jm}\in k[x_1,x_2,\dots,x_{n-1}]$ for all $0\leq m\leq t_j$ and $f_{jt_j}\neq 0$ for all $j$. 
        
        Now, $d_n(\rho(x_j))=\rho(d_n(x_j))$ for all $ 1\leq j\leq i$. Using \Cref{eqn2.3}, we have
 	\begin{align*}
        d_n(\sum_{m=0}^{t_j}f_{jm}x_n^m)&=\rho(h_j(x_1,x_2,\dots,x_i))\\
          &=\rho(\alpha_j x_1^{u_{j1}}x_2^{u_{j2}}\dotsm x_i^{u_{ji}}+O_j),
    \end{align*}
     for all $1\leq j\leq i$. Furthermore, 
     
    \begin{multline}\label{eqn2.2}
       d_n(f_{jt_{j}})x_n^{t_j}+d_n(f_{jt_j-1})x_n^{t_j-1}+\dots+d_n(f_{j0})+g_{n-1}(x_{n-1})\bigg(\sum_{m=1}^{t_j}mf_{jm}x_n^{m-1}\bigg)\\
 	=\alpha_j\rho(x_1)^{u_{j1}}\rho(x_2)^{u_{j2}}\dotsm\rho(x_i)^{u_{ji}}+\rho(O_j).
    \end{multline}
    Note that $f_{jt_j}\neq 0$, for all $1\leq j\leq i$. Substituting the expression of $\rho(x_j)$ from \Cref{eqn2.1} to \Cref{eqn2.2}, we get that
    \begin{align}
 	t_j\geq u_{j1}t_1+u_{j2}t_2+\dots+u_{ji}t_i, 
    \end{align} 
    for all $1\leq j\leq i$.
    
    Since $A_{d_i}={[u_{jq}]}_{1\leq j,q\leq i}$,  and it satisfies condition ($\ast$) of \Cref{$*$}, then it follows that $t_j=0$ for $1\leq j\leq i$. 
    
    Thus, $\rho(x_j)\in k[x_1,x_2,\dots,x_{n-1}]$ for all $1\leq j\leq i$. Now, similarly repeating the above process, we get that $\rho(x_j)\in k[x_1,x_2,\dots,x_{i}]$ for all $1\leq j\leq i$. 
    
    Since, $\rho^{-1}\in Aut(R_n)_{d_n}$. Replacing $\rho$ by $\rho^{-1}$ in the above argument, we get that $\rho^{-1}(x_j)\in k[x_1,x_2,\dots,x_i]$ for all $1\leq j\leq i$. Thence, $\rho|_{R_i}\in Aut(R_i)_{d_i}$. 
    \end{proof}

     \begin{remark}
     The above lemma still holds if $g_j(x_j)\in k[x_j]$ for all $i\leq j\leq n-1$.
    \end{remark}

    Next, we prove our main result.
    
    \begin{thm}\label{thm-isotropy}
     Let $i,n$, $R_i$, $R_n$, $d_i$ and $d_n$ be as above. Let $\deg_{x_i}g_i(x_i)\leq l$ for some $l\in \mathbb{N} \cup \{\infty\}$. Suppose $d_i$ satisfies the following conditions:
     \begin{enumerate}
         \item[(i)] If $d_i(r)=g(x_i)$ for some $r\in R_i$ and $g(x_i)\in k[x_i]$ with $\deg_{x_i}g(x_i)\leq l$, then $r\in k$ and $g=0$, 
         \item[(ii)] there exists a matrix $A_{d_i}$ corresponding to $d_i$ satisfying the condition $(\ast)$, as in \Cref{$*$}, and
         \item[(iii)] $Aut(R_i)_{d_i}=\{\text{Id}\}$. 
     \end{enumerate} 
     
     Then $Aut(R_n)_{d_n}=\{(\ x_1,x_2,\dots,x_n+c)|c\in k\}.$
     \end{thm}

    We first prove a few lemmas.

    \begin{lem}\label{x_n lemma}
    Let $i,n$, $R_i$, $R_n$, $d_i$ and $d_n$ be as above. Let $\deg_{x_i}g_i(x_i)\leq l$ for some $l\in \mathbb{N} \cup \{\infty\}$. Suppose $d_i$ satisfies conditions (i), (ii) and (iii) of \Cref{thm-isotropy}. Let $\rho\in Aut(R_n)_{d_n}$. 
    Then 
    \begin{enumerate}
        \item $\rho(x_{i+1})=x_{i+1}+\beta$ where $\beta \in k$,
        \item $\rho(x_j)\in R_{n-1}$ for all  $i+2\leq j \leq n-1$, and $\rho(x_n)=\alpha x_n+f_{n0}(x_1,\ldots ,x_{n-1})$ where $\alpha\in k^{\ast}$ and $f_{n0}(x_1,\ldots ,x_{n-1})\in R_{n-1}$.
    \end{enumerate}
    \end{lem}

    \begin{proof}
        We have $\rho\in$ Aut$(R_n)_{d_n}$. Since there exists a matrix corresponding to $d_i$ satisfying ($\ast$), it follows from \Cref{restriction_to_R_i} that $\rho |_{R_i}\in$ Aut$(R_i)_{d_i}$. Then from condition (iii), we have $\rho(x_j)=x_j$ for all $1\leq j\leq i$. 

    Also, from the condition $(i)$ on $d_i$ and \Cref{lemma no units}, it follows that: if  $f \in R_j$ be such that $d_j(f)=h(x_j) \in k[x_j]$, for any $i+1\leq j\leq n$, then $f\in k$ and $h=0$. We will use this fact throughout this proof.

    \textbf{(1)} Let $\rho(x_{i+1})=\sum_{j=0}^{t} f_j x_n^j$ where $f_j\in R_{n-1}$ and $f_t\neq 0$. From $d_n(\rho(x_{i+1}))$ $=\rho(d_n(x_{i+1}))$ we get that 
    \begin{multline}\label{eqn2.4}
        d_n(f_t)x_n^t+\cdots+d_n(f_0)+g_{n-1}(x_{n-1})(tf_tx_{n}^{t-1}+\cdots+f_1)=\rho(g_i(x_i))=g_i(x_i). 
    \end{multline}
    Comparing the coefficients of $x_n^l$, for all $l=0,\ldots,t$, we get the following equations
    \begin{align}
        \label{eqn2.5} d_n(f_t)&=0\\
        \label{eqn2.6} d_n(f_{t-1})+tf_tg_{n-1}(x_{n-1})&=0\\
        & \vdots  \nonumber \\
        \label{eqn2.7} d_n(f_1)+2f_2 g_{n-1}(x_{n-1})&=0\\
        \label{eqn2.8} d_n(f_0)+f_1g_{n-1}(x_{n-1})&=g_i(x_i).
    \end{align}
    Suppose $t\geq 2$. From \Cref{eqn2.5}, it follows that $f_t\in k^{\ast}$. Using that in \Cref{eqn2.6} we get that $d_n(f_{t-1}+tf_tx_n)=0$, which furthermore implies that $f_{t-1}+tf_tx_n\in k$, which is a contradiction. Hence $t\leq 1$.

    Now, suppose $t=1$. 
    Then using Equations (\ref{eqn2.5}) and (\ref{eqn2.8}) we get that $f_1\in k^{\ast}$ and $f_0+f_1x_n-x_{i+1}\in k$. Suppose $n=i+1$. Then it follows that $f_1=1$ and $f_0\in k$. Hence we are done. Now suppose $n>i+1$. Then from $f_0+f_1x_n-x_{i+1}\in k$ we get a contradiction. 
    
Suppose $t=0$ and $n=i+1$. Then we have $d_n(f_0)=g(x_i)$. Thus, $f_0\in k$ and $g_i(x_i)=0$, which is a contradiction.


Suppose that $n>i+1$ and $\rho(x_{i+1})\in R_{n-1}$.
Repeating the above procedure we can show that $\rho(x_{i+1})\in R_{i+1}$. Let $\rho(x_{i+1})=\sum_{j=0}^{t} f_j x_{i+1}^j$ where $f_j\in R_i$ and $f_t\neq 0$. Using condition (i) and arguing similarly as above, we can show that $t\leq 1$. If $t=0$ then $d_n(f_0)=g_i(x_i)\in k[x_{i}]$. Note that $d_n(f_0)=d_i(f_0)$ and $\deg_{x_i}g_i(x_i)\leq l$. Now, from condition (I), we have $f_0\in k$ and $g_i(x_i)=0$. This is a contradiction to the fact that $g_i\notin k$. Hence $t=1$, and we get that 
    \begin{align}
        \label{eqn2.9} d_n(f_1)&=0\\ 
        \label{eqn2.9'} d_n(f_0)+f_1g_i(x_i)&=g_i(x_i).  
    \end{align}
    From \Cref{eqn2.9} it follows that $f_1\in k^{\ast}$. Then from \Cref{eqn2.9'}, we have $d_n(f_0)=g_i(x_i)(1-f_1)\in k[x_i]$ with $\deg_{x_i}(d_n(f_0))\leq l$. Hence from condition (i) we get that $f_0\in k$, that is, $d_n(f_0)=0$. Also, as $g_i\neq 0$, we have $f_1=1$. Thus we can write $\rho(x_{i+1})=x_{i+1}+\beta$ for some $\beta \in k$. 
    
    \textbf{(2)} Note that here we have $n>i+1$. The proof of (2) is divided in the following two parts, namely (A) and (B). 
    
    \textbf{(A)} We first show by induction that for all $i+2\leq j \leq n$, $\rho(x_{j})=f_{j1}x_n+f_{j0}$ where $f_{j1}\in k$ and $f_{j0}\in R_{n-1}$. 
    
    Let $\rho(x_{i+2})=\sum_{m=0}^{t} f_{m} x_{n}^m$ where $f_t\neq 0$ and $f_m \in R_{n-1}$ for all $m$. If possible, let $t\geq 2$. From $d_n(\rho(x_{i+2}))=\rho(d_n(x_{i+2}))$ we get that $d_n(f_t)=0$ and $d_n(f_{t-1})+tf_tg_{n-1}(x_{n-1})=0$. Then $f_t\in k^{\ast}$ and $f_{t-1}+tf_tx_n\in k$, which is a contradiction. Hence $t\leq 1$. Now, if $t=1$, then $\rho(x_{i+2})=f_1x_n+f_0$ where $f_1\in k^{\ast}$ and $f_0\in R_{n-1}$, and if $t=0$ then $\rho(x_{i+2})=f_0\in R_{n-1}$. 
    

    Next, we assume that for some $j\in 
    \{i+1,\ldots,n-1\}$ we have $\rho(x_{j})=f_1x_n+f_0$ where $f_1\in k$ and $f_0\in R_{n-1}$.

Let $\rho(x_{j+1})=\sum_{m=0}^{s} h_mx_n^m$ where $h_s\neq 0$ and $h_m\in R_{n-1}$ for all $m$. Then from $d_n(\rho(x_{j+1}))=\rho(d_n(x_{j+1}))$ we get that 
     \begin{equation}\label{eqn2.10}
        d_n(h_s)x_n^s+\cdots+d_n(h_0)+g_{n-1}(x_{n-1})(sh_sx_{n}^{t-1}+\cdots+h_1)=g_j(f_1x_n+f_0). 
    \end{equation}
    Assume that $s\geq 2$. We will consider two cases.
    
    Case I. Suppose that $f_1=0$.  Then we get that $d_n(h_s)=0$ and $d_n(h_{s-1})+sh_sg_{n-1}(x_{n-1})=0$. Thus $h_s\in k^{\ast}$ and $h_{s-1}+sh_sx_n\in k$, which is a contradiction. 
    
    Case II. Now, suppose if $f_1\in k^{\ast}$. Note that the highest degree coefficient of $x_n$ in $g_j(f_1x_n+f_0)$ is in $k^{\ast}$. Also, $s\geq \deg_{x_n}(g_j(f_1x_n+f_0))$. Now, comparing the coefficient of $x_n^s$ from \Cref{eqn2.10}, we get that $d_n(h_s)\in k$. Then $h_s\in k^{\ast}$ and $d_n(h_s)=0$. Thence, we get $s-1\geq \deg_{x_n}(g_j(f_1x_n+f_0))$. Now, comparing the coefficient of $x_n^{s-1}$ we get that $d_n(h_{s-1})+sh_sg_{n-1}(x_{n-1})\in k$. Thus  
    $d_n(h_{s-1}+sh_sx_n) \in k$. Then $h_{s-1}+sh_sx_n \in k$, which is a contradiction.


    Hence, we get that $s\leq 1$. Note that $g_j\in k[x_j]\setminus k$. Suppose $s=1$, then comparing coefficient of $x_n$ from \Cref{eqn2.10} we have $d_n(h_1)\in k$, which further implies that $h_1\in k^{\ast}$. Now, if $s=0$, then $\rho(x_{j+1})=h_0\in R_{n-1}$. Thus we can write $\rho(x_{j+1})=h_1x_n+h_0$ where $h_1\in k$ and $h_0\in R_{n-1}$.

    \textbf{(B)} From (A) we have, for all $i+2\leq j \leq n$, $\rho(x_{j})=f_{j1}x_n+f_{j0}$ where $f_{j1}\in k$ and $f_{j0}\in R_{n-1}$. Now we show that $\rho(x_j)=f_{j0}\in R_{n-1}$ for all $i+2\leq j\leq n-1$ and $\rho(x_n)=\alpha x_n+f_{n0}$ where $\alpha \in k^{\ast}$ and $f_{n0}\in R_{n-1}$. 

    Recall that $\rho(x_n)=f_{n1}x_n+f_{n0}$ where $f_{n1}\in k$ and $f_{n0}\in R_{n-1}$. Suppose $f_{n1}=0$. Then $d_n(\rho(x_n))=\rho(d_n(x_n))$ implies that $d_n(f_{n0})=g_{n-1}(f_{(n-1)1}x_{n}+f_{(n-1)0})$. As $g_{n-1}\in k[x_{n-1}]\setminus k$ and $d_n(f_{n0})\in R_{n-1}$, we get that $f_{(n-1)1}=0$. Now, using $d_n(\rho(x_j))=\rho(d_n(x_j))$ for all $i+2\leq j\leq n-1$, inductively, we can show that $\rho(x_j)=f_{j0}\in R_{n-1}$. But this contradicts the fact that $\rho$ is an automorphism as $x_n\notin$ Image($\rho$). Hence $f_{n1}\neq 0$, i.e., $\rho(x_n)=f_{n1}x_n+f_{n0}$ where $f_{n1}\in k^{\ast}$ and $f_{n0}\in R_{n-1}$. 

    Again using $d_n(\rho(x_n))=\rho(d_n(x_n))$, we have 
    \begin{align}\label{eqn2.11}
        d_n(f_{n1}x_{n}+f_{n0})=\rho(g_{n-1}(x_{n-1}))=g_{n-1}(f_{(n-1)1}x_{n}+f_{(n-1)0}).
    \end{align}
    Thus $g_{n-1}(f_{(n-1)1}x_{n}+f_{(n-1)0})=f_{n1}g_{n-1}(x_{n-1})+d_n(f_{n0}) \in R_{n-1}$. Now, if $f_{(n-1)1}\in k^{\ast}$, we get a contradiction as $g_{n-1}(x_{n-1})\in k[x_{n-1}]\setminus k$. Hence $\rho(x_{n-1})=f_{(n-1)0}\in R_{n-1}$.  


    Similarly, from $d_n(\rho(x_{n-1}))=\rho(d_n(x_{n-1}))$, we get that $$d_n(f_{(n-1)0})=g_{n-2}(f_{(n-2)1}x_n+f_{(n-2)0}).$$ Again, if $f_{(n-2)1}\in k^{\ast}$, we get a contradiction. Hence $\rho(x_{n-2})=f_{(n-2)0}\in R_{n-1}$.

    Proceeding similarly, we get that $\rho(x_j)=f_{j0}\in R_{n-1}$ for all $i+2\leq j\leq n-1$. This completes the proof. 
    \end{proof}

    \begin{lem}\label{x_j_maps_till_x_j}
    Let $i,n$, $R_i$, $R_n$, $d_i$ and $d_n$ be as above. Let $\deg_{x_i}g_i(x_i)\leq l$ for some $l\in \mathbb{N} \cup \{\infty\}$. Suppose $d_i$ satisfies conditions (i), (ii), and (iii) of \Cref{thm-isotropy}. Let $\rho\in Aut(R_n)_{d_n}$. 
    Then, for all $i+1\leq j \leq n$, $\rho(x_j)=\alpha_{j}x_j+f_{j}(x_1,\ldots ,x_{j-1})$ where $\alpha_{j}\in k^{\ast}$ and $f_{j}(x_1,\ldots ,x_{j-1})\in R_{j-1}$.
    \end{lem}


     \begin{proof} 
     Note that if $n=i+1$ or $n=i+2$, the lemma follows from \Cref{x_n lemma}. So we may assume that $n>i+2$. 
     
     We have $\rho\in$ Aut$(R_n)_{d_n}$. From condition (ii) of \Cref{thm-isotropy}, \Cref{restriction_to_R_i} and condition (iii) of \Cref{thm-isotropy}, we have $\rho(x_j)=x_j$ for all $1\leq j\leq i$. Also, from \Cref{x_n lemma}, we have:
     \begin{itemize}
        \item $\rho(x_{i+1})=x_{i+1}+\beta$ for some $\beta \in k$, 
        \item $\rho(x_j)=f_{j0}\in R_{n-1}$ for all $i+2\leq j\leq n-1$, and
        \item $\rho(x_n)=f_{n1}x_n+f_{n0}$ where $f_{n1}\in k^{\ast}$ and $f_{n0}\in R_{n-1}$.
      \end{itemize}

    We will prove that $\rho(x_{n-1})=f_{(n-1)1}x_{n-1}+f_{(n-1)0}$ where $f_{(n-1)1}\in k^{\ast}$ and $f_{(n-1)0}\in R_{n-2}$, and $\rho(x_j)\in R_{n-2}$ for all $i+2\leq j\leq n-2$. Then the rest of the proof follows similarly.  

    \textbf{(C)} 
    We first show by induction that for all $i+2\leq j \leq n-1$, $\rho(x_{j})=f_{j1}x_{n-1}+f_{j0}$ where $f_{j1}\in k$ and $f_{j0}\in R_{n-2}$. The proof is similar to that of (2)(A) of \Cref{x_n lemma}. 

    Note that $\rho(x_{i+2})\in R_{n-1}$. Let $\rho(x_{i+2})=\sum_{m=0}^{t} f_{m} x_{n-1}^j$ where $f_t\neq 0$ and $f_m \in R_{n-2}$ for all $m$. If possible, let $t\geq 2$. From $d_n(\rho(x_{i+2}))=\rho(d_n(x_{i+2}))=g_{i+1}(x_{i+1}+\beta)$, where $\beta\in k$, we get that $d_n(f_t)=0$ and $d_n(f_{t-1})+tf_tg_{n-2}(x_{n-2})=0$. Then $f_t\in k^{\ast}$ and $f_{t-1}+tf_tx_{n-1}\in k$, which is a contradiction. Hence $t\leq 1$. Now, if $t=1$, then it follows that $\rho(x_{i+2})=f_1x_{n-1}+f_0$ where $f_1\in k^{\ast}$ and $f_0\in R_{n-2}$, and if $t=0$ then $\rho(x_{i+2})=f_0\in R_{n-2}$. Hence $\rho(x_{i+2})=f_1x_{n-1}+f_0$ where $f_1\in k$ and $f_0\in R_{n-2}$. 

    Proceeding inductively, as in (2)(A), we get that $\rho(x_j)=f_{j1}x_{n-1}+f_{j0}$ where $f_{j1}\in k$ and $f_{j0}\in R_{n-2}$, for all $i+2\leq j \leq n-1$. 

    \textbf{(D)} From (C) we have, for all $i+2\leq j \leq n-1$, $\rho(x_{j})=f_{j1}x_{n-1}+f_{j0}$ where $f_{j1}\in k$ and $f_{j0}\in R_{n-2}$. Now we show that $\rho(x_j)=f_{j0}\in R_{n-2}$ for all $i+2\leq j\leq n-2$ and $\rho(x_{n-1})=f_{(n-1)1} x_{n-1}+f_{(n-1)0}$ where $f_{(n-1)1}\in k^{\ast}$ and $f_{(n-1)0}\in R_{n-2}$. (The proof is similar to that of (2)(B)). 


    Recall that $\rho(x_{n-1})=f_{(n-1)1}x_{n-1}+f_{(n-1)0}$ where $f_{(n-1)1} \in k$ and $f_{(n-1)0}\in R_{n-1}$, and $\rho(x_{n-2})=f_{(n-2)1}x_{n-1}+f_{(n-2)0}$ where $f_{(n-2)1}\in k$ and $f_{(n-2)0}\in R_{n-2}$. Suppose $f_{(n-1)1}=0$. From $d_n(\rho(x_{n-1}))=\rho(d_n(x_{n-1}))$, we get that $d_n(f_{(n-1)0})=g_{n-2}(f_{(n-2)1}x_{n-1}+f_{(n-2)0})$. Since $d_n(f_{(n-1)0})\in R_{n-2}$ and $g_{n-2}(x_{n-2})\in k[x_{n-2}]\setminus k$, it follows that $f_{(n-2)1}=0$. Similarly, we get that $f_{j1}=0$ for all $i+2\leq j\leq n-2$, i.e., $\rho(x_j)\in R_{n-2}$ for all $i+2\leq j\leq n-2$. But this contradicts the fact that $\rho|_{R_{n-1}}$ is an injective map. Hence $f_{(n-1)1}\neq 0$, i.e., $\rho(x_{n-1})=f_{(n-1)1}x_{n-1}+f_{(n-1)0}$ where $f_{(n-1)1}\in k^{\ast}$ and $f_{(n-1)0}\in R_{n-2}$.

    Again by using $d_n(\rho(x_{n-1}))=\rho(d_n(x_{n-1}))$, we get that $$f_{(n-1)1}g_{n-2}(x_{n-2})+d_n(f_{(n-1)0})=g_{n-2}(f_{(n-2)1}x_{n-1}+f_{(n-2)0}).$$ If $f_{(n-2)1}\neq 0$, then we get a contradiction as $f_{(n-1)1}g_{n-2}+d_n(f_{(n-1)0})\in R_{n-2}$. Hence $f_{(n-2)1}=0$, i.e., $\rho(x_{n-2})=f_{(n-2)0}\in R_{n-2}$. Proceeding inductively we get that $\rho(x_j)=f_{j0}\in R_{n-2}$ for all $i+2\leq j\leq n-2$.

    Thus we have proved our claim that $\rho(x_{n-1})=f_{(n-1)1}x_{n-1}+f_{(n-1)0}$ where $f_{(n-1)1}\in k^{\ast}$ and $f_{(n-1)0}\in R_{n-2}$, and $\rho(x_j)=f_{j0}\in R_{n-2}$ for all $i+2\leq j\leq n-2$. 

    Repeating this process, (as in (C) and (D)), successively for $\rho(x_{n-2})$, $\rho(x_{n-3})$, $\ldots ,\rho(x_{i+2})$ we get the desired result. 
    \end{proof}

    \begin{lem}\label{final lemma}
    Let $i,n$, $R_i$, $R_n$, $d_i$ and $d_n$ be as above. Let $\deg_{x_i}g_i(x_i)\leq l$ for some $l\in \mathbb{N} \cup \{\infty\}$. Suppose $d_i$ satisfies conditions (i), (ii), and (iii) of \Cref{thm-isotropy}. Let $\rho\in Aut(R_n)_{d_n}$. Then $\rho(x_n)=x_n+c$ where $c\in k$, and $\rho(x_j)=x_j$ for all $1\leq j \leq n-1$.
    \end{lem}

    \begin{proof}
        Note that as before we have $\rho(x_j)=x_j$ for all $1\leq j\leq i$.

    By \Cref{x_n lemma}, $\rho(x_{i+1})=x_{i+1}+\beta$ where $\beta\in k$. Suppose $n=i+1$ then we are done. Thus, we may assume $n>i+1$.
        
    Recall that from \Cref{x_n lemma} and \Cref{x_j_maps_till_x_j},  we have $\rho(x_{i+1})=x_{i+1}+\beta$ and  $\rho(x_j)=\alpha_{j}x_j+f_{j}(x_1,\ldots ,x_{j-1})$ where $\beta\in k$, $\alpha_{j}\in k^{\ast}$ and $f_{j}(x_1,\ldots ,x_{j-1})\in R_{j-1}$ for all $i+2\leq j\leq n$.

    From $d_n(\rho(x_{i+2}))=\rho(d_n(x_{i+2}))$, we get that 
    
    \begin{align}\label{eqn2.12}
        \alpha_{i+2}\, g_{i+1}(x_{i+1})+d_n(f_{i+2})=g_{i+1}(x_{i+1}+\beta).
    \end{align}
        
    Then $d_n(f_{i+2})\in k[x_{i+1}]$. Since $f_{i+2}\in R_{i+1}$, we get that $f_{i+2}\in k$, and thus $d_n(f_{i+2})=0$. Then \Cref{eqn2.12} becomes $\alpha_{i+2}\, g_{i+1}(x_{i+1})=g_{i+1}(x_{i+1}+\beta)$. Since $g_{i+1}(x_{i+1})\in k[x_{i+1}]\setminus k$, we get that $\alpha_{i+2}=1$. Substituting $\alpha_{i+2}=1$, we have $g_{i+1}(x_{i+1})=g_{i+1}(x_{i+1}+\beta)$. Thus $\beta=0$. So, $\rho(x_{i+1})=x_{i+1}$ and $\rho(x_{i+2})=x_{i+2}+f_{i+2}$ where $f_{i+2}\in k$. 

    Suppose, for some $j\in \{i+1,\ldots,n-1\}$, we have $\rho(x_s)=x_s$ for all $1\leq s\leq j-1$ and $\rho(x_j)=x_j+f_{j}$ where $f_j\in k$. Then we show that $\rho(x_{j})=x_{j}$  and $\rho(x_{j+1})=x_{j+1}+f_{j+1}$ where $f_{j+1}\in k$. 

    Note that we have $\rho(x_{j+1})=\alpha_{j+1}x_{j+1}+f_{j+1}$ where $\alpha_{j+1}\in k^{\ast}$ and $f_{j+1}\in R_{j}$. Now, from $d_n(\rho(x_{j+1}))=\rho(d_n(x_{j+1}))$, we get that 
        
    \begin{align}\label{eqn2.13}
        \alpha_{j+1}\, g_{j}(x_{j})+d_n(f_{j+1})=g_{j}(x_{j}+f_j).
    \end{align}
        
    Then $d_n(f_{j+1})\in k[x_{j}]$. Since $f_{j+1}\in R_{j}$, we get that $f_{j+1}\in k$, and thus $d_n(f_{j+1})=0$. Then \Cref{eqn2.13} becomes $\alpha_{j+1}\, g_{j}(x_j)=g_{j}(x_{j}+f_j)$. Since $g_{j}(x_j)\in k[x_{j}]\setminus k$, we get that $\alpha_{j+1}=1$. Substituting $\alpha_{j+1}=1$, we get that $g_{j}(x_{j})=g_{j}(x_{j}+f_j)$. Thus $f_j=0$. So, $\rho(x_{j})=x_{j}$ and $\rho(x_{j+1})=x_{j+1}+f_{j+1}$ where $f_{j+1}\in k$.     

    By the above induction we get that $\rho(x_n)=x_n+c$ where $c\in k$, and $\rho(x_j)=x_j$ for all $1\leq j \leq n-1$. This completes the proof.
    \end{proof}

    \textbf{Proof of \Cref{thm-isotropy}}: By \Cref{final lemma}, we get that $$Aut(R_n)_{d_n}\subseteq\{(x_1,x_2,\dots,x_n+c)|c\in k\}.$$ It is easy to check that the other containment also holds. Hence, the proof follows.\qed

    We next recall an important result due to Shamsuddin \cite{sam}, a proof of which is also provided in \cite[Theorem 13.2.1]{n94}.

    \begin{thm}{\cite{sam}}\label{sam result}
        Let $R$ be a ring containing $\mathbb{Q}$ and let $d$ be a simple derivation of $R$. Extend the derivation $d$ to a derivation $\tilde{d}$ of the polynomial ring $R[t]$ by setting $\tilde{d}(t)=at+b$ where $a,b\in R$. Then the following two conditions are equivalent.
    \begin{enumerate}
        \item $\tilde{d}$ is simple.
        \item There exist no elements $r$ of $R$ such that $d(r) = ar + b$.
    \end{enumerate}
    \end{thm}

    \begin{cor}\label{d_n simple}
        Let $n$, $R_i$, $R_n$, $d_i$ and $d_n$ be as on \Cref{thm-isotropy}. Let $\deg_{x_i}g_i(x_i)\leq l$ for some $l\in \mathbb{N} \cup \{\infty\}$. Suppose $d_i$ satisfies the following condition: If $d_i(r)=g(x_i)$ for some $r\in R_i$ and $g(x_i)\in k[x_i]$ with $\deg_{x_i}g(x_i)\leq l$, then $r\in k$ and $g=0$. If $d_i$ is a simple derivation, then $d_{n}$ is a simple derivation. 
    \end{cor}

    \begin{proof}
        We just prove it for $d_{i+1}$, then the rest of the proof follows similarly using \Cref{lemma no units}. Recall that $d_{i+1}=d_i+g_i(x_{i})\partial_{x_{i+1}}$ where $\deg_{x_i}g_i(x_i)\leq l$ and $g_{i}(x_i)\in k[x_i]\setminus k$. Thus $d_{i+1}(x_{i+1})=g_i(x_i)$. Now, let there exists an $r\in R_i$ such that $d_i(r)=g_i(x_i)$. Since $\deg_{x_i}g_i(x_i)\leq l$ we get that $r\in k$ and thus $g_i(x_i)=0$. But $g_i(x_i)\notin k$ and hence we get a contradiction. Then the derivation $d_{i+1}$ is simple by \Cref{sam result}.
    \end{proof}

   \begin{remark}
        In the above corollary, in order to prove that $d_{i+1}$ is simple, it is enough to assume that $d_i$ is simple, and $g_i(x_i)$ does not belong to the image of $d_i$. 
    \end{remark}

    In \cite{MP17}, the authors proved that for any simple derivation of $R_2$, the corresponding isotropy group is trivial. In view of this result, we have the following corollaries.

\begin{cor}\label{isotropy-for-2-var}
         Let $n$, $R_i$,$R_n$, $d_i$ and $d_n$ be as on \Cref{thm-isotropy}. 
         Let $i=2$ and assume that $\deg_{x_2}g_2(x_2)\leq l$ for some $l\in \mathbb{N} \cup \{\infty\}$. 
         Suppose $d_2$ satisfies the following conditions:
         \begin{enumerate}
             \item[(i)] If $d_2(r)=g(x_2)$ for some $r\in R_2$ and $g(x_2)\in k[x_2]$ with $\deg_{x_2}g(x_2)\leq l$, then $r\in k$ and $g=0$, and 
             \item[(ii)] there exists a matrix $A_{d_2}$ corresponding to $d_2$ satisfying the condition $(\ast)$, as in \Cref{$*$}.
        \end{enumerate} 
        If $d_2$ is a simple derivation of $R_2$, then $d_n$ is simple and the isotropy group is $Aut(R_n)_{d_n}=\{(\ x_1,x_2,\dots,x_n+c)|c\in k\}$. 
    \end{cor}

    \begin{proof}
        Since $d_2$ is a simple, $Aut(R_2)_{d_2}=\{\text{Id}\}$ by \cite[Theorem 1]{MP17}. Thus, from \Cref{thm-isotropy}, the isotropy group is  $Aut(R_n)_{d_n}=\{(x_1,x_2,\dots,x_n+c)|c\in k\}$. The derivation $d_n$ is simple by \Cref{d_n simple}.
    \end{proof}

    \begin{cor}\label{isotropy-new-example}
        Let $m_1$ and $m_2$ be positive integers such that $m_2$ does not divide $m_1$. Let $d_2= (1+{x_1}^{m_2}x_2)\partial_{x_1}+ x_1^{m_1}\partial_{x_2}$. Let $n>2$ and let $d_n:=d_2+g_2(x_2)\partial_{x_{i+1}}+\dots+g_{n-1}(x_{n-1})\partial_{x_n}$,
        where $g_j(x_j)\in k[x_j]\backslash k$ such that $2\leq j\leq n-1$.
        Then $d_n$ is simple and $Aut(R_n)_{d_n}=\{(\ x_1,x_2,\dots,x_n+c)|c\in k\}.$ 
    \end{cor}

    \begin{proof} For $l=\infty$, the condition $(i)$ of \Cref{isotropy-for-2-var} is satisfied, see \Cref{lemma1}.

        We now prove that the assumption $(ii)$ of \Cref{isotropy-for-2-var} is also satisfied. We claim that the matrix $A_{d_2}= \begin{pmatrix} m_2 & 1\\ m_1 & 0 \end{pmatrix}$ corresponding to $d_2$ satisfies the condition $(\ast)$, as in \Cref{$*$}. Consider $(A_{d_2}-I)Y\leq 0$ where $I$ is the $2\times 2$ identity matrix and $Y=  \begin{pmatrix} t_1 \\ t_2\end{pmatrix}$ where $t_1, t_2\geq 0$. Then we have \begin{equation}\label{2.9.1} (m_2-1)t_1+t_2 \leq 0, \end{equation} 
             \begin{equation}\label{2.9.2}  m_1t_1-t_2\leq 0. \end{equation} 
    Adding \Cref{2.9.1} and \Cref{2.9.2}, we get that $(m_1+m_2-1)t_1\leq 0$, which implies that $t_1=0$. Now putting $t_1=0$ in \Cref{2.9.1}, we get that $t_2=0$. 
    Thus, the claim holds.

The derivation $d_2$ is simple by \cite[Theorem 4.1 and Theorem 6.1]{k14}. Now, the result follows from \Cref{isotropy-for-2-var}.
    \end{proof}

    \begin{remark}
        In case of $m_1<m_2$, the fact that $d_n$ is simple is known (see \cite[Theorem 5.1]{k14}).   
    \end{remark}

     \begin{cor}\label{isotropy-new-example-1}
      Let $m_1$ and $m_2$ be positive integers such that $m_1\geq 2$. 
      Let $d_2=(1-{x_2}^{m_2} x_1)\partial_{x_1}+{x_1}^{m_1}\partial_{x_2}$ be a $k$-derivation of $R_2$. 
      Let $n>2$ and
      let $d_n:=d_2+g_i(x_i)\partial_{x_{i+1}}+g_{i+1}(x_{i+1})\partial_{x_{i+2}}+\dots+g_{n-1}(x_{n-1})\partial_{x_n}$, where $g_i(x_i)\in k[x_i]\setminus k$ with $\deg_{x_i}g_i(x_i)\leq 1$ and $g_j(x_j)\in k[x_j]\setminus k$ for all $i+1\leq j\leq n-1$, be a $k$-derivation of $R_n$.
        Then $d_n$ is simple and $Aut(R_n)_{d_n}=\{(\ x_1,x_2,\dots,x_n+c)|c\in k\}.$ 
    \end{cor}

    \begin{proof}
        From \cite[Proposition 2.3]{MMS25}, we have that the condition (i) of \Cref{isotropy-for-2-var} is satisfied for $d_2$ for $l=1$. 

        We now prove that assumption $(ii)$ of \Cref{isotropy-for-2-var} is also satisfied. We claim that the matrix $A_{d_2}= \begin{pmatrix} 1 & m_2\\ m_1 & 0\end{pmatrix}$ corresponding to $d_2$ satisfies the condition $(\ast)$, as in \Cref{$*$}.
        Consider $(A_{d_2}-I)Y\leq 0$ where $I$ is the $2\times 2$ identity matrix and $Y= \begin{pmatrix} t_1 \\ t_2\end{pmatrix}$ where $t_1, t_2\geq 0$. Then we have 
        \begin{equation}
        \label{2.12.1} m_2t_2\leq 0, 
        \end{equation}
        \begin{equation}
        \label{2.12.2}m_1t_1-t_2\leq 0.
        \end{equation} 
        The Equations (\ref{2.12.1}) and (\ref{2.12.2}) imply that $t_1=0$ and $t_2=0$. Thus, the claim holds.

        The derivation $d_2$ is simple by \cite[Theorem 2.6]{ap23}. Now, the result follows from \Cref{isotropy-for-2-var}.
    \end{proof}

    We also obtain the following result proved in \cite[Theorem 3.4]{Y24}. 

    \begin{cor}\label{isotropy-new-example-2}
        Let $m\geq 2$ be a positive integer. Let $d_2= ({x_1}^{m}x_2+g)\partial_{x_1}+ x_1^{m-1}\partial_{x_2}$ such that $g\in k[x_1]$ with $\deg_{x_1}g\leq m$ and gcd$(x_1,g)$$=1$. 
        
        Let $n>2$ and let $d_n:=d_2+g_2(x_2)\partial_{x_{i+1}}+\dots+g_{n-1}(x_{n-1})\partial_{x_n}$,
        where $g_j(x_j)\in k[x_j]\backslash k$ for all $2\leq j\leq n-1$.
        Then $d_n$ is simple and $Aut(R_n)_{d_n}=\{(\ x_1,x_2,\dots,x_n+c)|c\in k\}$. 
    \end{cor}

    \begin{proof}
        From \cite[Lemma 3.2]{k12}, we get that the condition (i) of \Cref{isotropy-for-2-var} is satisfied for $d_2$ for $l=\infty$.

        We now prove that assumption $(ii)$ of \Cref{isotropy-for-2-var} is also satisfied. We claim that the matrix $A_{d_2}= \begin{pmatrix} m & 1\\ m-1 & 0\end{pmatrix}$ corresponding to $d_2$ satisfies the condition $(\ast)$, as in \Cref{$*$}.
        Consider $(A_{d_2}-I)Y\leq 0$ where $I$ is the $2\times 2$ identity matrix and $Y= \begin{pmatrix} t_1 \\ t_2\end{pmatrix}$ where $t_1, t_2\geq 0$. 

        Then we have 
        \begin{equation}
        \label{2.12.3} (m-1)t_1+t_2\leq 0, 
        \end{equation}
        \begin{equation}
        \label{2.12.4}(m-1)t_1-t_2\leq 0.
        \end{equation} 
        The \Cref{2.12.3} imply that $t_1=0$ and $t_2=0$.

Also, the derivation $d_2$ is simple by \cite[Theorem 4.1]{k14}. Thus by \Cref{isotropy-for-2-var}, $d_n$ is simple and $Aut(R_n)_{d_n}=\{(\ x_1,x_2,\dots,x_n+c)|c\in k\}$.
    \end{proof}

  
\section{The derivation $y^{m_1}\partial_x+(1+y^{m_2}x)\partial_y$}

Let $k$ be a field of characteristic zero. In this section, we provide a class of simple derivations $d_2$ on $R_2=k[x,y]$ which satisfy the condition $(i)$ of \Cref{thm-isotropy} for $l=\infty$. 

Let $m_1$ and $m_2$ be positive integers such that $m_2$ does not divide $m_1$. Let $d:=y^{m_1}\partial_x+(1+y^{m_2}x)\partial_y$ denote the $k$-derivation of the two variables polynomial ring $k[x,y]$. 

For any $h(y)\in k[y]$, we denote $\partial_yh(y)$ by $h'(y)$.

\begin{lem}\label{lemma1}
Let $d$ be the derivation of $R=k[x,y]$, as defined above. Let $f \in R$ be such that $d(f)=g(x)$, where $g(x)\in k[x]$. Then $f\in k$ and $g(x)=0$.
\end{lem}

\begin{proof}
Let $g(x)\in k[x]$. Suppose $f(x,y)\in k[x,y]$ be such that $d(f)=g(x)$. If $f=0$, we are done. Thus we assume that $f\neq 0$ and write $f=\sum_{i=0}^{t}f_ix^i$, where $f_i\in k[y]$, for $0\leq i\leq t$, and $f_t\neq 0$. 
We will prove that for $t=0$, $f\in k$, and for $t\geq 1$ there does not exist any such $f$.  
        
If $t=0$ then $f\in k[y]$. So, we have that $f=f_0(y)$, where $f_0(y)\in k[y]$. Now, $d(f_0(y))=g(x)$ implies $(y^{m_2}x+1)f_0'=g(x)$. Thus, $f_0'(y)=0$. Hence, $f=f_0\in k$ and $g(x)=d_n(f_0)=0$.
        
Next, we assume that $t\geq 1$ and $g(x)=b_lx^l+b_{l-1}x^{l-1}+\dots+b_1x+b_0$, where $b_i\in k$ for all $i=1,\ldots,l$ and $b_l\neq 0$. From $d(f)=g(x)$ it follows that
\begin{equation}
\label{maineqn}
\sum_{i=1}^{t}if_iy^{m_1}x^{i-1}+\sum_{i=0}^{t}y^{m_2}f'_ix^{i+1}+\sum_{i=0}^{t}f'_ix^i= b_lx^l+b_{l-1}x^{l-1}+\dots+b_1x+b_0.
\end{equation}
Thus $t+1\geq l$. Let $t+1=l$. From \Cref{maineqn} it follows that $y^{m_2}f_t'=b_l\in k^{\ast}$, which is absurd. Next, suppose that $t=l$. From \Cref{maineqn}, it follows that $f'_t=0$ and $y^{m_2}f'_{t-1}(y)+f'_t(y)=b_l$. Thence $y^{m_2}f'_{t-1}(y)=b_l$, which is not possible. Hence $t+1\geq l+2$, i.e., $t\geq l+1$. 
        
Since $t-1\geq l$, i.e., $l\in \{t-1,t-2,\dots,1,0\}$, we rewrite $g(x)=b_{t-1}x^{t-1}+b_{t-2}x^{t-2}+\dots+b_1x+b_0$, where $b_i\in k $ for $0\leq i\leq t-1$. Similarly, rewriting \Cref{maineqn} we have
	\begin{equation}\label{maineqn2}
		\sum_{i=1}^{t}if_iy^{m_1}x^{i-1}+\sum_{i=0}^{t}y^{m_2}f'_ix^{i+1}+\sum_{i=0}^{t}f'_ix^i=b_{t-1}x^{t-1}+b_{t-2}x^{t-2}+\dots+b_1x+b_0.
	\end{equation}
Note that the above equation also holds for $g=0$.    
Now, comparing the coefficients of $x^i$ for $i =t+1,\ldots ,0$ in \Cref{maineqn2}, we get the following equations
		\begin{align}
			\label{eqn1}	y^{m_2}f'_t(y)&=0,\\
			\label{eqn2}	y^{m_2}f'_{t-1}(y)+f'_t(y)&=0,\\
			\label{eqn3}  ty^{m_1}f_t(y)+y^{m_2}f'_{t-2}(y)+f'_{t-1}(y)&=b_{t-1},\\
			\hspace{3cm} & \vdots  \nonumber
			\\
			     \label{eqn4} 3y^{m_1}f_3(y)+y^{m_2}f'_{1}(y)+f'_{2}(y)&=b_2,\\
			   \label{eqn5} 2y^{m_1}f_2(y)+y^{m_2}f'_{0}(y)+f'_{1}(y)&=b_1,\\
			  \label{eqn6} y^{m_1}f_1(y)+f'_{0}(y)&=b_0,
		\end{align} respectively. 
Since $t\geq 1$, we always have Equations (\ref{eqn1}) and (\ref{eqn2}), then it follows that $f_t\in k^{\ast}$ and $f_{t-1}\in k$. Let us denote 
    \begin{align}
         \lambda_t:=f_t \in k^{\ast} \,\,\,\,\, \text{and} \,\,\,\,\, \lambda_{t-1}:=f_{t-1}\in k.
    \end{align}
   Suppose $t=1$. Then from \Cref{eqn6}, we get that $f_1=0$, which is a contradiction. Hence, $t\neq 1$. Hence $t\geq 2$. 

Now, from \Cref{eqn3}, we get that
		\begin{align*}
			y^{m_2}f'_{t-2}(y)=b_{t-1}-t\lambda_ty^{m_1},
		\end{align*}which implies that $b_{t-1}=0$. Thus
	\begin{align}\label{Eqn1}
		y^{m_2}f'_{t-2}(y)=-t\lambda_ty^{m_1}
	\end{align}
    
		Suppose $m_1<m_2$ then it follows from \Cref{Eqn1} that $y$ divides $\lambda_{t}$, which is a contradiction. Hence, from now on, we consider that $m_1\geq m_2$.
        
		Since $m_2 \nmid m_1$, we have $m_1>m_2$. Then the \Cref{Eqn1} implies that
		\begin{align}\label{Eqn1'}
			f'_{t-2}(y)=-t\lambda_ty^{m_1-m_2}.
		\end{align}
        
       Next, on comparing coefficient of $x^{t-2}$ from \Cref{maineqn2}, we get that
		\begin{align}\label{eqn(t-2)}
		(t-1)y^{m_1}f_{t-1}(y)+y^{m_2}f'_{t-3}(y)+f'_{t-2}(y)=b_{t-2}.
	\end{align} Substituting $f'_{t-2}$, from \Cref{Eqn1'} in \Cref{eqn(t-2)}, we have
		\begin{align*}
			y^{m_2}f'_{t-3}(y)=b_{t-2}-(t-1)\lambda_{t-1}y^{m_1}+t\lambda_ty^{m_1-m_2},
		\end{align*}which further implies that $b_{t-2}=0$. Thus \Cref{eqn(t-2)} becomes
	\begin{align}\label{Eqn2'}
		y^{m_2}f'_{t-3}(y)=-(t-1)\lambda_{t-1}y^{m_1}+t\lambda_ty^{m_1-m_2}.
	\end{align} Suppose that $m_2<m_1<2m_2$. Then from \Cref{Eqn2'} we get that $y$ divides $\lambda_t$, which is a contradiction. If $m_1>2m_2$, then \Cref{Eqn2'} becomes
    \begin{align}
       f'_{t-3}(y)=-(t-1)\lambda_{t-1}y^{m_1-m_2}+t\lambda_ty^{m_1-2m_2}.
    \end{align}


Next, we claim the following.

\textbf{Claim:} Suppose $m_1>(j-1)m_2$ for some $2\leq j \leq t-1$, then we have 
		\begin{align}\label{claim-old}
        y^{m_2}f'_{t-(j+1)}=\alpha_{j0}(y)y^{m_1}+\alpha_{j1}(y)y^{m_1-m_2}+\dots+\alpha_{j(j-1)}(y)y^{m_1-(j-1)m_2},
		\end{align} 

        
        where $\alpha_{ji}(y)\in k[y]$ for $0\leq i\leq j-2$ and $\alpha_{j(j-1)}(y)\in k^{\ast}$.
        
        \vspace{2mm}
        \textbf{Proof of Claim:} Suppose $j=2$, i.e., $m_1>m_2$ , then as in \Cref{Eqn2'} we have that $y^{m_2}f'_{t-3}(y)=-(t-1)\lambda_{t-1}y^{m_1}+t\lambda_ty^{m_1-m_2}$, where $a_{20}:=-(t-1)\lambda_{t-1}\in k[y]$ and $a_{21}:= t\lambda_t\in k^{\ast}$. 
		Thus, the claim holds for $j=2$.


        Next, suppose for some $2\leq s-1\leq t-2$ the claim holds.  

        Now, assume that $m_1>(s-1)m_2$. From \Cref{maineqn2}, equating the coefficient of $x^{t-s}$, we get that

        
        \begin{align}\label{eqn7}
            (t-(s-1))y^{m_1}f_{t-(s-1)}(y)+y^{m_2}f'_{t-(s+1)}(y)+f'_{t-s}(y)=b_{t-s}. 
        \end{align}

       Note that $m_1>(s-1)m_2>(s-2)m_2$. By the induction hypothesis, we have

       \begin{multline}\label{eqn8-old}
            y^{m_2}f'_{t-s}(y)=\alpha_{(s-1)0}(y)y^{m_1}+\alpha_{(s-1)1}(y)y^{m_1-m_2}+\dots \\\dots+\alpha_{(s-1)(s-2)}(y)y^{m_1-(s-2)m_2},
        \end{multline}
        where $\alpha_{(s-1)i}(y)\in k[y]$ for $0\leq i\leq s-3$ and $\alpha_{(s-1)(s-2)}(y)\in k^{\ast}$.


        Now, since $m_1>(s-1)m_2$, we can write \Cref{eqn8-old} as 
        
         \begin{multline}\label{eqn9-old}
            f'_{t-s}=\alpha_{(s-1)0}(y)y^{m_1-m_2}+\alpha_{(s-1)1}(y)y^{m_1-2m_2}+\dots\\ \dots+\alpha_{(s-1)(s-2)}(y)y^{m_1-(s-1)m_2}.
        \end{multline}

        

        Substituting the value of $f'_{t-s}(y)$ from \Cref{eqn9-old} in \Cref{eqn7}, we get that 

        \begin{multline}\label{eqn10-old}
             y^{m_2}f'_{t-(s+1)}(y)=b_{t-s}-(t-(s-1))y^{m_1}f_{t-(s-1)}(y)\\
             -(\alpha_{(s-1)0}(y)y^{m_1-m_2}+\alpha_{(s-1)1}(y)y^{m_1-2m_2}+\dots+\alpha_{(s-1)(s-2)}(y)y^{m_1-(s-1)m_2}).
        \end{multline}



        Since $m_1>(s-1)m_2$, it follows that $b_{t-s}=0$.
        Thus we can write \Cref{eqn10-old} as
        \begin{align}\label{eqn11-old}
            y^{m_2}f'_{t-(s+1)}=\alpha_{s0}(y)y^{m_1}+\alpha_{s1}(y)y^{m_1-m_2}+\dots+\alpha_{s(s-1)}(y)y^{m_1-(s-1)m_2},
        \end{align}
        where $\alpha_{s0}(y)=-(t-(s-1))f_{t-(s-1)}(y)$ and $\alpha_{si}(y)=-\alpha_{(s-1)(i-1)}(y)$ for all $1\leq i\leq s-1$.


        Note that $\alpha_{si}(y)\in k[y]$ for $0\leq i\leq s-2$ and $\alpha_{s(s-1)}(y)\in k^{\ast}$. Hence, the claim follows for $s$.   \qed
        

        Thus, if $jm_2>m_1>(j-1)m_2$ for some $2\leq j\leq t-1$, then from \Cref{claim-old} it follows that $y$ divides $\alpha_{j(j-1)}\in k^{\ast}$, which  is a contradiction. Hence, we get that $m_1>(t-1)m_2$. Now from \Cref{eqn6} we have 
        \begin{align}\label{eqn12}
            y^{m_1}f_1(y)+f'_0(y)=b_0. 
        \end{align}
        Also, as $m_1>(t-1)m_2>(t-2)m_2$, from \Cref{claim-old} we have 
        \begin{multline}\label{eqn13}
            f'_0(y)=\alpha_{(t-1)0}(y)y^{m_1-m_2}+\alpha_{(t-1)1}(y)y^{m_1-2m_2}+\dots\\\ldots+\alpha_{(t-1)(t-2)}(y)y^{m_1-(t-1)m_2},
        \end{multline}
        where $\alpha_{(t-1)(t-2)}\in k^{\ast}$. 

        Since $y$ divides $f'_0$ in \Cref{eqn13}, it follows from \Cref{eqn12} that $b_0=0$. Furthermore, substituting \Cref{eqn13} in \Cref{eqn12}, we get that 
        \begin{multline}
            y^{m_1}f_1(y)=-\alpha_{(t-1)0}(y)y^{m_1-m_2}-\alpha_{(t-1)1}(y)y^{m_1-2m_2}-\dots\\\ldots-\alpha_{(t-1)(t-2)}(y)y^{m_1-(t-1)m_2}.
        \end{multline}
        
        Then $y$ divides $\alpha_{(t-1)(t-2)}\in k^{\ast}$, which is a contradiction. Thus, we finally get that $t< 1$. This completes the proof.  
	\end{proof}

	\begin{remark}
		When $m_1=m_2$, $d(r)=1$ for $r=-\frac{x^2}{2}+y$. Hence, the above lemma does not hold for the case when $m_1=m_2$.
	\end{remark}

\section*{Declaration of interests} The authors declare that they have no known competing financial interests or personal relationships that could have appeared to influence the work reported in this paper.

\section*{Acknowledgements} The first-named author would like to thank the Department of Science and Technology of India for the INSPIRE Faculty fellowship grant IFA23-MA 197 and IIT Indore for the Young Faculty Seed Grant IITIIYFRSG 2024-25/Phase-V/04R. The second author would like to thank IIT Indore for the Young Faculty Research Grant (No. IITIIYFRSG 2024-25/Phase-VII/05). The third-named author is financially supported by the SRF grant from CSIR India, Sr. No. 09/1022(16075)/2022-EMR-I.

\end{document}